\documentclass[10pt,reqno]{amsart}
\usepackage{mathtools}
\usepackage{a4wide}
\usepackage{dsfont}            
\usepackage{amssymb,amsmath,amsfonts}
\usepackage{mathrsfs}
\usepackage{graphicx} 
\usepackage{xcolor}          
\usepackage[cp1250]{inputenc}
\usepackage[all]{xy}
\newtheorem{proposition}{Proposition}[section]
  \newtheorem{theorem}[proposition]{Theorem}
  \newtheorem{corollary}[proposition]{Corollary}
  \newtheorem{lemma}[proposition]{Lemma}
\theoremstyle{definition}
  \newtheorem{definition}[proposition]{Definition}
  \newtheorem{remark}[proposition]{Remark}
   
  \newtheorem{example}[proposition]{Example}
\allowdisplaybreaks
\newcommand{\cst}{\ifmmode\mathrm{C}^*\else{$\mathrm{C}^*$}\fi}

\newcommand{\CC}{\mathbb{C}}
\newcommand{\RR}{\mathbb{R}}
                                      
\newcommand{\GG}{\mathbb{G}}

\newcommand{\tens}{\otimes}
\newcommand{\vtens}{\,\bar{\otimes}\,}
\newcommand{\id}{\mathrm{id}}

\newcommand{\I}{\mathds{1}}

\newcommand{\sM}{\mathsf{M}}
\newcommand{\sN}{\mathsf{N}}

\newcommand{\sA}{\mathsf{A}}
\newcommand{\sX}{\mathsf{X}}

\newcommand{\hh}[1]{\widehat{#1}}

\newcommand{\ww}{\mathrm{W}}

\newcommand{\Ww}{\mathds{W}}
\newcommand{\wW}{\text{\reflectbox{$\Ww$}}\:\!} 

\newcommand{\staru}{\,\overline{*}\,}

\DeclareMathOperator{\C}{C}
\DeclareMathOperator{\B}{B}

\DeclareMathOperator{\Mor}{Mor}
\DeclareMathOperator{\M}{M}

\DeclareMathOperator{\Linf}{\mathnormal{L}^\infty\;\!\!}
\DeclareMathOperator{\Ltwo}{\mathnormal{L}^2\;\!\!}

\numberwithin{equation}{section}

\def\labelitemi{$\blacktriangleright$}
\author{Pawe{\l} Kasprzak}
\address{Department of Mathematical Methods in Physics, Faculty of Physics, University of Warsaw, Poland}
\email{pawel.kasprzak@fuw.edu.pl}
\title[Shifts of group-like projections]{Shifts of group-like projections and contractive idempotent functionals for locally compact quantum groups}


\begin{document}

\begin{abstract}A one to one  correspondence between shifts of group-like projections on a locally compact quantum group $\GG$ which are preserved by the scaling group and contractive idempotent functionals on the dual $\hh\GG$ is established. This is a generalization of the Illie-Spronk's   correspondence   between contractive idempotents in the Fourier-Stieltjes algebra of a locally compact group $G$ and cosets of open subgroups of $G$.  We also  establish a one to one  correspondence between non-degenerate, integrable, $\GG$-invariant ternary rings of operators $\sX\subset \Linf(\GG)$, preserved by the scaling group and contractive idempotent functionals on $\GG$.  Using our results we  characterize  coideals in  $\Linf(\hh\GG)$ admitting an atom preserved by the scaling group in terms of idempotent states on $\GG$. We also establish a one to one  correspondence between integrable coideals in   $\Linf(\GG)$ and group-like projections in  $\Linf(\hh\GG)$ satisfying an  extra mild condition. Exploiting this correspondence  we give examples of group like projections  which are not preserved by the scaling group.
\end{abstract}
\maketitle
                                

\newlength{\sw}
\settowidth{\sw}{$\scriptstyle\sigma-\text{\rm{weak closure}}$}
\newlength{\nc}
\settowidth{\nc}{$\scriptstyle\text{\rm{norm closure}}$}
\newlength{\ssw}
\settowidth{\ssw}{$\scriptscriptstyle\sigma-\text{\rm{weak closure}}$}
\newlength{\snc}
\settowidth{\snc}{$\scriptscriptstyle\text{\rm{norm closure}}$}
\renewcommand{\labelitemi}{$\bullet$}
\section*{Introduction} The interest in projections in the Fourier-Stieltjes algebra $\B(G)$ of a locally compact groups dates back to Cohen \cite{Cohen}. Assuming that $G$ is abelian, Cohen proved that  projections in $\B(G)$ are in 1-1 correspondence with  elements of  the boolean ring generated by the cosets of open subgroups of $G$. Later  Host \cite{Host} proved this for all locally compact groups. Remembering that $\B(G)$ is  identified with the dual Banach space of the universal group $\C^*$-algebra $\C^*(G)$, a more specific class of projections in $\B(G)$ can be considered. In this respect Illie and  Spronk  \cite[Theorem 2.1]{IS} proved that idempotent states correspond to characteristic functions of  open subgroups of $G$ whereas contractive idempotent functionals correspond  to characteristic functions of cosets of open subgroups. 

Moving on to the realm of locally compact quantum groups,   Faal and the author of the present paper extended the first half of Illie-Spronk result proving a 1-1 correspondence between group-like projections on $\Linf(\GG)$ preserved by the scaling group, and  idempotent states on the dual locally compact quantum group $\hh\GG$,  see \cite{FallKasp}. One of the main result of the present paper provides the quantum group  extension of the second half the Illie-Spronk result by giving characterization of contractive idempotent functionals on $\hh\GG$ in terms of   {\it shifts of group-like projections} preserved by the scaling group. Needless to say, that  if $G$ is locally compact group then     group-like projections in $\Linf(G)$ and their shifts are characteristic functions of open subgroups of $G$ and their cosets respectively. More generally, a  group-like projection $P\in\Linf(\GG)$  corresponds to an open quantum subgroup of $\GG$ if and only if it is in the center of $\Linf(\GG)$, see \cite{KKS}.

Throughout the paper $\GG$ denotes a locally compact quantum group in the sens of Kustermans and Vaes \cite{KVvN}. The dual of $\GG$ is denoted by $\hh\GG$.  The von Neumann  and the $\C^*$-algebra  assigned to $\GG$ will be denoted by $\Linf(\GG)$ and $\C_0(\GG)$ respectively, the latter being strongly dense in the former. The von Neumann algebra of $\GG$ is equipped with a comultiplication $\Delta:\Linf(\GG)\to\Linf(\GG)\vtens\Linf(\GG)$, unitary coinverse $R:\Linf(\GG)\to\Linf(\GG)$ and the scaling group $\tau_t:\Linf(\GG)\to\Linf(\GG)$, each nicely restricting to a morphism, antiautomorphism and continuous action of $\RR$ on $\C_0(\GG)$ respectively. The coinverse $S$ is a possibly unbounded map densely defined on $\Linf(\GG)$ and it admits the polar decomposition $S = R\circ\tau_{-\frac{i}{2}}$. The left and right Haar weight on $\Linf(\GG)$ are denoted $\varphi$ and $\psi$ respectively  and  the GNS map assigned to $\psi$ is denoted  $\eta: D(\eta)\to\Ltwo(\GG)$. The modular group of automorphisms of $\Linf(\GG)$ assigned to $\psi$ and $\varphi$ are denoted $(\sigma^\psi_t)_{t\in\mathbb{R}}$ and $(\sigma^\phi_t)_{t\in\mathbb{R}}$ respectively. If $\psi(\I)<\infty$ then we say that $\GG$ is a compact quantum group (CQG) and write $\C(\GG)$ instead $\C_0(\GG)$ in this case. Referring to the dual locally compact quantum group $\hh\GG$ we write $\hh\Delta$, $\hh\tau$ etc.  The universal $\C^*$-algebra assigned to $\GG$ is denoted $\C_0^u(\GG)$, and the corresponding versions of $\Delta,R$ and $\tau$ are decorated with the superscript $u$, e.g. $\Delta^u\in\Mor(\C_0^u(\GG),\C_0^u(\GG)\otimes \C_0^u(\GG))$.
The multiplicative unitary $\ww\in\B(\Ltwo(\GG)\otimes\Ltwo(\GG))$ is characterized and defined by  
\begin{equation}\label{multunit1}(\id\otimes\omega)(\ww)\eta(x) = \eta(\omega*x)\end{equation} whenever $x\in D(\eta)$ and $\omega\in\B(\Ltwo(\GG))_*$, where for $x\in\Linf(\GG)$ and $\omega\in\Linf(\GG)_*$ we write $\omega*x = (\id\otimes\omega)(\Delta(x))$ and $x*\omega = (\omega\otimes\id)(\Delta(x))$. We have $\Delta(x) = \ww(x\otimes\I)\ww^*$ for all $x\in\Linf(\GG)$ and $\hh\Delta(y) = \sigma(\ww^*(\I\otimes y)\ww)$ where $\sigma:\Linf(\hh\GG)\vtens\Linf(\hh\GG)\to \Linf(\hh\GG)\vtens\Linf(\hh\GG)$ is the flipping map: $\sigma(u\otimes v) = v\otimes u$. We will use the same symbol to denote the flipping  $\sigma:\sM\vtens\sN\to\sN\vtens\sM$ where $\sM$, $\sN$ are arbitrary von Neumann algebras. We  will often view $\Linf(\GG)$  and  $\Linf(\hh\GG)$  as  von Neumann algebras acting on $\Ltwo(\GG)$.  In fact we have  $\ww\in\Linf(\hh\GG)\vtens\Linf(\GG)$ and 
\begin{align*}\Linf(\GG)&=\overline{\{(\omega\otimes\id)(\ww):\omega\in\B(\Ltwo(\GG))_*\}}^{\sigma-\textrm{weak}}\\
\Linf(\hh\GG)&=\overline{\{(\id\otimes\omega)(\ww):\omega\in\B(\Ltwo(\GG))_*\}}^{\sigma-\textrm{weak}}\end{align*} where ${\sigma-\textrm{weak}}$ in the superscript position means the ${\sigma-\textrm{weak}}$ closure of the  considered set.
 We shall often view $\ww$ as an element of $\M(\C_0(\hh\GG)\tens\C_0(\GG))$. The half-lift $\wW\in\M(\C_0(\hh\GG)\tens\C_0^u(\GG))$ of $\ww$ will play an important role. It can be proved that
\begin{equation}\label{multunit2}(\id\otimes\omega)(\wW)\eta(x) = \eta(\omega*x)\end{equation}
for all $x\in D(\eta)$ and $\omega\in\C_0^u(\GG)^*$.
Writing $\Lambda\in\Mor(\C_0^u(\GG),\C_0(\GG))$ for the canonical reducing surjection we have $(\id\otimes \Lambda)(\wW) =\ww$. We say that $\GG$ is coamenable if $\Lambda$ is an isomorphism, or equivalently if $\C_0(\GG)$ admits a character $\varepsilon:\C_0(\GG)\to\mathbb{C}$ such that $(\id\otimes\varepsilon)(\ww) = \I$. 
 By coideal of $\Linf(\GG)$ we mean a von Neumann subalgebra $\sN\subset\Linf(\GG)$ such that $\Delta(\sN)\subset\Linf(\GG)\vtens\sN$. We shall often use the property, called Podle\'s condition \cite[Corollary 2.7]{gpext} (satisfied by all coideals) 
 \begin{equation}\label{Podcond}\sN =\{(\nu\otimes\id)(\Delta(x)):\nu\in\Linf(\GG)_*,x\in\sN\}^{\sigma-\textrm{cls}}\end{equation}   where ${\sigma-\textrm{cls}}$ in the superscript position means the $\sigma$-weak closure of the linear  span of  the   considered set.    Given a coideal we define the {\it codual coideal} $\widetilde{\sN}= \sN'\cap\Linf(\hh\GG)$. Since $\widetilde{\widetilde{\sN}}=\sN$ the correspondence $\sN\mapsto\widetilde{\sN}$ establishes a bijective correspondence between coideals in $\Linf(\GG)$ and in $\Linf(\hh\GG)$, for the details see \cite{embed}. We say that $\sN$ is integrable if the set of elements   $x\in\sN^+$ satisfying $\psi(x)<\infty$ is strongly dense in $\sN^+$. The completion of $\eta(\sN\cap D(\eta))$ in $\Ltwo(\GG)$ is denoted  $\Ltwo(\sN)$. 
For more about the role of integrability condition in the theory of locally compact quantum groups see \cite{int}. In particular it was prove there  that $\sN$ is integrable iff $\sN\cap D(\eta)\neq \{0\}$. 

A state on $\omega\in\C_0^u(\GG)$ such that its convolution square  $\omega*\omega\coloneqq(\omega\otimes\omega)\circ \Delta^u$ is equal to $\omega$ is called an {\it idempotent state}.  Idempotent states on quantum groups has  been intensely studied by now, see e.g. \cite{FSO}, \cite{SaS}, \cite{Salmi_Survey}, \cite{coid_sub_st}, \cite{FallKasp}. Note that if $\mu\in\C_0^u(\GG)^*$, $\mu\neq 0$ (here $\mu$ is a functional, not necessarily a state)  satisfies $\mu*\mu=\mu$ and $\|\mu\|\leq 1$ then $\|\mu\|=1$. Such functionals are called {\it contractive idempotent functionals}, see \cite{NSSS} for their theory. As shown in the present paper a contractive idempotent functional $\mu\in\C_0^u(\GG)^*$ is assigned with a left shift of a group-like projection (see Definition \ref{shiftsdef}): $Q = (\id\otimes\mu)(\wW)\in\Linf(\hh\GG)$. Remarkably,  shifts of group-like projections appeared recently in \cite{UnPr}, \cite{UnPr1} in the context of uncertainty principle  on locally compact quantum groups. Actually,   group-like projections considered there were assumed to have finite Haar weight, and as noted in   \cite{PSL}, they are in 1-1 correspondence with normal idempotent states on $\GG$. Our  Definition \ref{shiftsdef} of a shift of a group-like projection  and    the one used in \cite{UnPr1} are compared in Remark  \ref{compar}.  

The paper is written as follows. In the end of Introduction we fix further notation and prove some auxiliary facts. In  Section \ref{Sec2} we show that if $\GG$ is a locally compact quantum group and     $\sN\subset \Linf(\GG)$ is an integrable coideal, then  the projection $P_\sN$ onto $\Ltwo(\sN)\subset\Ltwo(\GG)$ is a group-like projection  in $\Linf(\hh\GG)$  and the coideal $\sN\subset\Linf(\GG)$  is preserved by the scaling group of $\GG$ if and only if $P_\sN$ is preserved by the scaling group of $\hh\GG$. Using this result we give  examples of group-like projections   which are not preserved by the scaling group.  We  also note that the assignment $\sN\mapsto P_\sN$ is injective  and characterize group-like projections in $\Linf(\hh\GG)$  which are  of the form $ P_\sN$. If  $\GG$ is a compact quantum group then we prove that the map \[\{\textrm{coideals in} \Linf(\GG)\}\ni\sN\mapsto P_\sN\in\{\textrm{group-like projections in } \Linf(\hh\GG)\}\] is surjective.   We end Section \ref{Sec2}  with the definition of a left and a right shift of a group-like projection.  In Section \ref{Sec3} we show that given a contractive idempotent functional $\omega\in\C_0^u(\GG)^*$, the element $(\id\otimes\omega)(\wW)\in\Linf(\hh\GG)$ is a shift of a group-like projection   preserved by the scaling group $\hh\tau$. We derive from that the invariance of a contractive idempotent functionals under the scaling group. In Section \ref{Sec4} we prove  the converse of the result of Section \ref{Sec3}:  shifts of group-like projection preserved by the scaling group are all of the form described in Section \ref{Sec3}. In Section \ref{Sec5} we use our results to prove that a coideal $\widetilde\sN\subset\Linf(\hh\GG)$ admits an atom  (i.e. a non-zero central minimal projection $Q\in Z(\widetilde\sN)$)  preserved by the scaling group if and only if $\widetilde\sN$ is $\hh\GG$-generated by a group-like projection $P\in\widetilde\sN$ which is preserved by the scaling group and $Q$ is a left shift of $P$ (see Definition \ref{Ggener} for the concept of $\hh\GG$-generation).  Finally, in Section \ref{sec6} we establish a 1-1 correspondence between non-degenerate, integrable, $\GG$-invariant ternary rings of operators $\sX\subset \Linf(\GG)$, preserved by the scaling group and contractive idempotent functionals on $\GG$. 
Summarizing the main results of this paper   establish a 1-1 correspondence between:
\begin{itemize}
    \item idempotent contractive functionals  $\omega\in\C_0^u(\GG)^*$;
    \item shifts of group-like projections $Q\in\Linf(\hh\GG)$ preserved by the scaling group $\hh\tau$, where $\omega$ is assigned with $Q_\omega = (\id\otimes\omega)(\wW)$;
    \item non-degenerate, integrable, $\GG$-invariant ternary rings of operators $\sX\subset \Linf(\GG)$, preserved by the scaling group, where the corresponding projection $Q_\sX\in\Linf(\hh\GG)$ maps $\Ltwo(\GG)$ onto $\Ltwo(\sX)$. 
\end{itemize}

\begin{definition}\label{Ggener} Let $\GG$ be a locally compact quantum group, $\sN\subset\Linf(\GG)$   a coideal and $x\in\sN$. We say that $\sN$ is $\GG$-generated by $x$ if \begin{equation}\label{ggener}\sN = \{(\nu\otimes\id)(\Delta(x)):\nu\in\Linf(\GG)_*\}''.\end{equation}
\end{definition}
\begin{lemma}\label{leminvt}
Suppose that $\sN$ is $\GG$-generated by $x$. If $\tau_t(x)=x$ for all $t\in\RR$ then \begin{equation}\label{inv1}\tau_t(\sN) =\sN = \sigma^\psi_t(\sN) \textrm{\quad for all } t\in\RR.\end{equation} If $\sigma^\varphi_t(x) = x$ then \begin{equation}\label{inv2}\sigma^\varphi_t(\sN)  =\sN \textrm{\quad for all } t\in\RR.\end{equation} 
\end{lemma}
\begin{proof}
The proof of \eqref{inv1} follows from the identities 
\[\Delta\circ\tau_t = (\tau_t\otimes\tau_t)\circ\Delta = (\sigma^\varphi_t\otimes\sigma^\psi_{-t})\circ\Delta\] and \eqref{ggener}. 
The proof of \eqref{inv2} follows in turn from $\Delta\circ\sigma^\varphi_t = (\tau_t\otimes\sigma^\varphi_t)\circ\Delta$. For the identities used above we refer to \cite[Proposition 6.8]{KV}.
\end{proof}
The following result  will be used in Example \ref{exnoninv}. 
 \begin{lemma}\label{taupres}
 Let $\GG$ be a locally compact quantum group, $\pi:\Linf(\GG)\to\Linf(\GG)$ and $\rho:\Linf(\GG)\to\Linf(\GG)$ be automorphisms of $\Linf(\GG)$ such that $\Delta\circ \pi = (\pi\otimes\pi)\circ\Delta$ and $\Delta\circ \rho = (\rho\otimes\rho)\circ\Delta$ and let $x\in\Linf(\GG)$ be such that $(\pi\otimes\rho)(\Delta(x))\in\Delta(\Linf(\GG))$. Then $\pi(x) = \rho(x)$.
 \end{lemma}
 \begin{proof}
  Note that 
 \[(\Delta\otimes\id)((\pi\otimes\rho)(\Delta(x)) =(\id\otimes\Delta)((\pi\otimes\rho)(\Delta(x))\] which implies that
 \[(\pi\otimes\pi\otimes\rho)(\Delta^{(2)}(x)) = (\pi\otimes\rho\otimes\rho)(\Delta^{(2)}(x))\]
 thus \begin{equation}\label{delt2}(\id\otimes\pi\otimes\id)(\Delta^{(2)}(x)) = (\id\otimes\rho\otimes\id)(\Delta^{(2)}(x))\end{equation} where $\Delta^{(2)} = (\Delta\otimes\id)\circ\Delta$. Using \cite[Theorem 2.3]{FallKasp} (where the group $\GG\times\GG$ acts on $\Linf(\GG)$ by shifts from the left and the right) we obtain $x\in\{\omega*x*\mu:\omega,\mu\in\Linf(\GG)_*\}''$ and thus using \eqref{delt2} we get $\pi(x) = \rho(x)$.
 \end{proof}
 Let us finish this section with the following plausible  lemma.
 \begin{lemma}\label{invar}
 Let $\sM$ be a von Neumann algebra,  $P\in\sM$ a minimal central projection and $\alpha:\RR\to\mathrm{Aut}(\sM)$ a continuous   action of $\RR$ on $\sM$. Then $\alpha_t(P) = P$.
 \end{lemma}
 \begin{proof}
Using the minimality and centrality of $P$ we can see, that for every  $y\in\sM$ there exists $\varepsilon_y\in\mathbb{C}$ such that $yP =\varepsilon_y P$.  In particular $\alpha_t(P)P=\varepsilon_{\alpha_t(P)}P$ and since   $\alpha_t(P)P$ is a projection we have $\varepsilon_{\alpha_t(P)}\in\{0,1\}$. Using the continuity of the action $\alpha$ we conclude that the map $\mathbb{R}\ni t\mapsto\varepsilon_{\alpha_t(P)}\in\{0,1\}$ is continuous and since $\varepsilon_{\alpha_0(P)} = 1$ we get \begin{equation}\label{alpP}\alpha_t(P)P = P\end{equation} for all $t\in\mathbb{R}$. Applying $\alpha_{-t}$ to \eqref{alpP}    we conclude that $\alpha_{-t}(P)P = \alpha_{-t}(P)$ for all $t\in\mathbb{R}$ which together with \eqref{alpP} yields $\alpha_t(P) = P$. 
 \end{proof}
 
\section{Group-like projections and their shifts}\label{Sec2}
\begin{definition}\label{glprojdef}
Let $P\in\Linf(\GG)$ be a non-zero self-adjoint projection. We say that $P$ is a group-like projection if 
\begin{equation}\label{eq1}\Delta(P)(\I\otimes P) = P\otimes P.
\end{equation}
\end{definition}
Defining group-like projections one must choose between the identity \eqref{eq1} and its alternative 
\begin{equation}\label{eq2}
\Delta(P)(P\otimes\I) = P\otimes P.\end{equation} Our choice in Definition \ref{glprojdef} is related with the choice in the definition of  coideal (the alternative here would be $\Delta(\sN)\subset\sN\vtens\Linf(\GG)$)  as seen in Proposition \ref{coidglp}. Let us note that the unitary coinverse $R$ turns left coideals into  right coideal, and  group-like projections in the sens of \eqref{eq1} into those satisfying \eqref{eq2}.

\begin{lemma}\label{eq00}
If $P\in\Linf(\GG)$ is a group-like projection then $P\in\M(\C_0(\GG))$. If moreover $\GG$ is coamenable then $\varepsilon(P)=1$.
\end{lemma}
\begin{proof}Let  $V\in\M(\C_0(\GG)\otimes\mathcal{K}(\Ltwo(\GG))) $ be  the left regular representation
\[V = (\hh J\otimes\hh J)\sigma(\ww)(\hh J\otimes\hh J) \] where $\hh J:\Ltwo(\GG)\to\Ltwo(\GG)$ is the Tomita-Takesaki conjugation assigned to the right invariant Haar weight $\hh\psi$ on $\hh\GG$. 
Since 
  $\Delta(P) =  V(\I\otimes P)V^* \in\M(\C_0(\GG)\otimes\mathcal{K}(\Ltwo(\GG)))$ we have $P\otimes P = \Delta(P)(\I\otimes P)\in\M(\C_0(\GG)\otimes\mathcal{K}(\Ltwo(\GG)))$. This shows that    $P\in\M(\C_0(\GG))$. 
In particular if $\GG$ is coamenable,  then applying  $(\varepsilon\otimes\id)$   to both sides of  \eqref{eq1}  we get $P = P^2 = \varepsilon(P) P$, thus $\varepsilon(P)=1$.
\end{proof}

Let $\GG$ be a locally compact quantum group.
The main results of \cite{FallKasp} and \cite{coid_sub_st} establish a 1-1 correspondence between:
\begin{itemize}
    \item idempotent states $\omega\in\C_0^u(\GG)^*$;
    \item group-like projections $P\in\Linf(\hh\GG)$ preserved by the scaling group $\hh\tau$, where $\omega$ is assigned with $P_\omega = (\id\otimes\omega)(\wW)$;
    \item integrable coideals $\sN\subset\Linf(\GG)$ preserved by the scaling group $\tau$  , where the corresponding projection $P_\sN\in\Linf(\hh\GG)$ maps $\Ltwo(\GG)$ onto $\Ltwo(\sN)$. 
\end{itemize}
 In order to relate $\sN$ with $\omega$ and $P$ we shall also use notation $\sN_\omega$, $\sN_P$  etc. 
The following relations are scattered throughout \cite{FallKasp} and \cite{coid_sub_st} 
\begin{align} \label{relPN1} 
\sN_P &= \overline{\{(P\nu\otimes\id)(\ww):\nu\in\Linf(\hh\GG)_*\}}^{\sigma-  \textrm{weak}}\\\label{relPN2}
\sN_P &= \{x\in\Linf(\GG):xP = Px\}\\\label{relPN3}
\widetilde{\sN_P}& = \overline{\{(\nu\otimes\id)(\hh\Delta(P)):\nu\in\Linf(\hh\GG)_*\}}^{\sigma-\textrm{weak}}\\\label{relPN4}\widetilde{\sN_P}& = \{y\in\Linf(\hh\GG):\hh\Delta(y)(\I\otimes P) = y\otimes P\}
\end{align} 
where in \eqref{relPN1} we used the bimodule structure of $\Linf(\hh\GG)_*$ over $\Linf(\hh\GG)$,  i.e.  given $\nu\in\Linf(\hh\GG)_*$ and $x,y\in\Linf(\hh\GG)$  we define $\nu x, y\nu\in\Linf(\hh\GG)_*$:
\[
   (\nu x)(y) = \nu(xy),\quad (y\nu)(x) = \nu(xy).
\]
Note that \eqref{relPN3} implies that $P$ is ($\hh\tau$-invariant)  minimal central projection in $\widetilde{\sN_P}$. In Section \ref{Sec5} we shall prove that a coideal in $\Linf(\hh\GG)$ admitting a minimal central projection $Q$ preserved  by $\hh\tau$ must be of the form $\widetilde{\sN_P}$ for certain group-like projection $P\in\Linf(\hh\GG)$ preserved by $\hh\tau$. Moreover $Q$ is then a left shift of $P$ in the sens on Definition \ref{shiftsdef}.
\begin{lemma}\label{remR}
Let $P\in\Linf(\hh\GG)$ be a group-like projection preserved by the scaling group $\hh\tau$. Then $\hh R(P)=P$ and $\hh\Delta(P)(P\otimes\I) = P\otimes P$. 
\end{lemma}
\begin{proof}
 Let $\omega\in\C_0^u(\GG)^*$ be an idempotent state such that $P = P_\omega$. Remembering that $\omega$ is preserved by $R^u$  (see \cite{SaS}) and using  $(\hh R\otimes R^u)(\wW)=\wW$ together with $P=(\id\otimes\omega)(\wW)$ we get $\hh R(P) = P$. Applying $(\hh R\otimes\hh R)$ to the identity \eqref{eq1} we get \eqref{eq2}.
\end{proof}
Using \eqref{relPN3}, Lemma \ref{leminvt}, Lemma \ref{invar}, Lemma \ref{remR} and the relation $\sigma^{\hh\psi}_t = \hh R\circ\sigma^{\hh\varphi}_{-t}\circ \hh R$ we get
\begin{corollary}\label{corinvsig}
Let $P\in\Linf(\hh\GG)$ be a group-like projection preserved by $\hh\tau_t$. Then 
\begin{itemize}
    \item $\widetilde{\sN_P}$ is preserved by $\hat\tau$, $\sigma^{\hh\psi}$ and $\sigma^{\hh\varphi}$;
    \item $\sigma^{\hh\psi}_t(P)=P=\sigma^{\hh\varphi}_t(P)$  for all $t\in\RR$.
\end{itemize}
\end{corollary}

Let $\sN\subset \Linf(\GG)$ be an integrable coideal and   $P\in\B(\Ltwo(\GG)) $ the projection  onto $\Ltwo(\sN)$. In the next proposition we show, that $P$  is a group-like projection in $\Linf(\hh\GG)$. Let us emphasize that this result holds without the assumption that $\sN$ is preserved by $\tau$, c.f. \cite{FallKasp}. 
We denote \[\mathcal{T}_\psi  =  \{x\in  D(\eta)\cap D(\eta)^*:x  \textrm{ is  }  \sigma^\psi-\textrm{analytic  and}\,\,\sigma^\psi_{z}(x)\in D(\eta)\cap D(\eta)^* \textrm{ for all } z\in\CC\}.\] 
\begin{proposition}\label{coidglp}
Let $\GG$ be a locally compact quantum group, $\sN\subset \Linf(\GG)$   an integrable coideal and $P\in\B(\Ltwo(\GG))$ the orthogonal projection onto $\Ltwo(\sN)$. Then $P\in\Linf(\hh\GG)$   is a group-like projection and we have \begin{equation}\label{descript_N}\sN  = \overline{\{(P\omega\otimes\id)(\ww):\omega\in\Linf(\hh\GG)_*\}}^{\sigma-\textrm{weak}}.\end{equation} Moreover 
$\sN$ is preserved by the scaling group $\tau$  if and only if $P$ is preserved by $\hat\tau$.   
\end{proposition}
\begin{proof} The reasoning from the beginning of the proof of \cite[Theorem 4.19]{PSL} applies also in our case 
and yields that $P\in\Linf(\hh\GG)\cap\sN'$ (c.f. also the beginning of the proof of Theorem \ref{propTro}).  Using the equality  $(\id\otimes\omega)(\ww)\eta(x) = \eta((\id\otimes\omega)(\ww))$ for  $x\in D(\eta)\cap\sN$ and $\omega\in\Linf(
\GG)_*$ we get 
\begin{align*}(\id\otimes\omega)((\I\otimes P)\ww)\eta(x)&=\eta((\id\otimes \omega)((\I\otimes P)\Delta(x)))\\&=\eta((\id\otimes \omega)(\Delta(x)(\I\otimes P)))\\&=(\id\otimes\omega)(\ww(\I\otimes P))\eta(x),\end{align*}
where in the third equality we used $\Delta(x)\in\Linf(\GG)\vtens\sN$ and $P\in\sN'$. 
This shows that $(\I\otimes P)\ww(P\otimes\I) = \ww(P\otimes P)$ which in turn implies that $\ww^*(\I\otimes P)\ww(P\otimes\I) = P\otimes P$, i.e. $\hh\Delta(P)(\I\otimes P) = P\otimes P$. 

 Suppose that  $\sN$ is preserved by $\tau$.  Using \cite[Theorem  4.2]{coid_sub_st}  and \cite[Theorem 3.1]{FallKasp}  we conclude that $P$  is $\hh\tau$-invariant  (a direct proof can also be obtained as in the proof of Corollary \ref{corom}). 

For $x\in\sN\cap D(\eta)$  and   $y\in D(\eta)$ we have 
\[\ww(\eta(x)\otimes\eta(y)) = (\eta\otimes\eta)(\Delta(x)(\I\otimes  y).\]
In particular,  if $z\in D(\eta)$  then $(\omega_{\eta(z),\eta(x)}\otimes\id)(\ww)    =  (\psi\otimes\id)((z^*\otimes\I)\Delta(x))$.  Assuming that $a,b\in\mathcal{T}_\psi$  and putting $z^*  =  \sigma^\psi_{i}(a)b^*$  we see that     $(\omega_{\eta(z),\eta(x)}\otimes\id)(\ww)     = (\omega_{\eta(b),\eta(a)}\otimes\id)(\Delta(x))\in\sN$. Using Podle\'s condition for  $\sN$ we conclude that $\sN\subset \overline{\{(P\omega\otimes\id)(\ww):\omega\in\Linf(\hh\GG)_*\}}$.  
The converse inclusion is obtained  by noting that a functional    $P\omega$   can be approximated by linear combination  of those  of the form $\omega_{\eta(z),\eta(x)}$  where $x\in\sN\cap D(\eta)$  and  $z^*  =  \sigma^\psi_{i}(a)b^*$  for $a,b\in\mathcal{T}_\psi$.

Note that $\tau_t((P\omega\otimes\id)(\ww) )  = ( \hh\tau_{-t}(P)\hh\tau_t^*(\omega)\otimes\id)(\ww)$. In particular if 
 $P$  is preserved by $\hat\tau$ then using Equation  \eqref{descript_N}  we conclude the  $\tau$-invariance of $\sN$. 
\end{proof}
Using Equation \eqref{descript_N}   we see that $x\in\tilde{\sN}$ if and only if 
$\ww(P\otimes x)  =  (\I\otimes x)\ww(P\otimes\I)$  which is equivalent with $\hh\Delta(x)(\I\otimes P_\sN)  = x\otimes P_\sN$.  Together with the Podle\'s condition for $\tilde{\sN}$ the  latter implies the next corollary.
\begin{corollary}\label{corcod}
 Let $\sN$ be an integrable coideal.  Then 
 \[\widetilde{\sN} = \{x\in\Linf(\hh\GG):\hh\Delta(x)(\I\otimes P_\sN)  = x\otimes P_\sN\}.\] In particular $P_\sN$  is a minimal central projection in $\widetilde{\sN}$.  
\end{corollary}

\begin{example}\label{exnoninv}
Let us give two classes of examples of group-like projections not preserved by the respective  scaling group:
\begin{enumerate}
    \item 

Let $0<q<1$, $\GG = SU_q(2)$ and let $\sN$ be the von Neumann algebra assigned to a  Podle\'s sphere, see \cite{PodlesSphere}. In what follows we restrict our attention to   embeddable Podle\'s spheres, i.e.  we assume that $\sN$  is a coideal in $\Linf(\GG)$.  Since all idempotent states on $SU_q(2)$ are of Haar type (see \cite{FST}),  $\sN \neq \sN_\omega$ for all idempotent states $\omega$  on $SU_q(2)$ unless $\sN$ corresponds to the standard Podle\'s sphere (the one which is of quotient type).  In particular the group-like projection  $P_\sN\in\Linf(\hh{SU_q(2)})$ assigned to $\sN$ is not preserved by the scaling group unless it corresponds to the standard Podle\'s sphere.
\item
Let $\GG$ be a CQG  and let $\sN = \Delta(\Linf(\GG))\subset \Linf(\GG)\vtens\Linf(\GG^\textrm{op})$ be the $\GG\times\GG^{\textrm{op}}$-coideal as described in \cite[Section 6]{embed}. Suppose that $\tau^{\GG\times\GG^{\textrm{op}}}_t = \tau^\GG_t\otimes\tau^\GG_{-t}$ preserves $\sN$. Then using Lemma \ref{taupres} (with $\pi=\tau^\GG_t$ and $\rho = \tau^\GG_{-t}$) we get that $\tau_t = \id$ for all $t\in\mathbb{R}$ . In particular if $\GG$ has a non-trivial scaling group then  $\sN$  is not of the form $\sN_\omega$ for an idempotent state on $\GG\times\GG^\textrm{op}$ and thus the group-like projection  $P_\sN\in\Linf(\hh{\GG\times\GG^{\textrm{op}}})$ assigned to $\sN$ is not preserved by the scaling group. Note that the same can also be concluded using \cite[Proposition 3.5]{FranzLeeSkalski}, and conversely, the above reasoning yields an alternative proof of  \cite[Proposition 3.5]{FranzLeeSkalski}.
\end{enumerate}
\end{example} 

Using Proposition  \ref{coidglp}    we see that the assignment $\sN\mapsto P_\sN$ is injective. 
In the next theorem we characterize group-like projections in $\Linf(\hh\GG)$ which are of the form $P = P_\sN$ for an integrable coideal $\sN\subset\Linf(\GG)$. 

\begin{theorem}\label{thmCQGcoid}
 Let $P\in\Linf(\hh\GG)$ be a group-like projection. There exists an integrable coideal $\sN$ such that $P = P_{\sN}$ if and only there exists $a\in D(\eta)\setminus\{0\}$ such that $P\eta(a) = \eta(a)$. 
\end{theorem}
\begin{proof}
The implication $\implies$  is clear. For the converse implication consider 
\[\sM=\{x\in\Linf(\hh\GG):\hh\Delta(x)(\I\otimes P) = x\otimes P = (\I\otimes P)\hh\Delta(x).\}\]  It is easy to check that $\sM$ forms a coideal and  $P$ is a minimal central projection in $\sM$ (c.f. Corollary  \ref{corcod}). Let us consider the codual coideal 
\begin{equation}\sN=\{y\in\Linf(\GG):xy=yx \textrm{ for all } x\in\sM\}.\end{equation}
Using $\ww(P\otimes x)  =(\I\otimes x)\ww(P\otimes\I)$ we conclude that $(P\omega\otimes\id)(\ww)\in\sN$  for all $\omega\in\Linf(\hh\GG)_*$. In particular, since $P\eta(a)  = \eta(a)$, we have $(\omega\otimes\id)(\Delta(a))\in\sN$ for all $\omega\in\Linf(\hh\GG)_*$ (c.f. the third paragraph    of the proof of Proposition \ref{coidglp}). Using \cite[Remark 2.4]{FallKasp}  we conclude that $a\in\sN$ and thus $\sN$ is integrable.  Let $P_N$  be the projection corresponding to $\sN$. Using Corollary \ref{corcod}  we see that $P_N$ is a minimal central projection in $\tilde{\sN}$.  Since $PP_\sN\eta(x)  =  \eta(x)$. Thus $PP_\sN\neq 0$ and by the minimality of both we get  $P =  PP_\sN  = P_\sN $.  
\end{proof}
In the next theorem we show that the condition of Theorem \ref{thmCQGcoid} is satisfied for all group-like projections $P\in\Linf(\hh\GG)$ if $\GG$ is a compact quantum group. 
\begin{theorem}\label{thmCQGcoid1}
Let $\GG$ be a CQG and $P\in\Linf(\hh\GG)$ a group-like projection. Then there exists a unique coideal $\sN\subset\Linf(\GG)$ such that $P  = P_\sN$. 
\end{theorem}
\begin{proof}
We define  $\sN\subset\Linf(\GG)$ as in the proof of Theorem \ref{thmCQGcoid}.  Clearly $\sN$ is integrable.  Using Corollary  \ref{corcod} we see that   $P_\sN\in\tilde\sN$.  Applying $(\hh\varepsilon\otimes\id)$ to the   identity  $\hh\Delta(P_\sN)(\I\otimes P)  =  P_{\sN}\otimes P$. We conclude that $P_{\sN}P  =  P$ thus $P_{\sN}\subset P$.  Since $P$ is minimal we have $P = P_{\sN}$. 
 \end{proof}

 Before defining shifts of group-like projection we shall prove the following auxiliary lemma.

\begin{lemma}\label{basic_lem}
Let $Q\in\Linf(\GG)$ be a non-zero self-adjoint projection such that 
\begin{equation}\label{gleq}(\I\otimes Q)\Delta(Q) = y\otimes Q.\end{equation} Then $(\I\otimes y)\Delta(y) = y\otimes y$. Moreover $\psi(Q) = \psi(y)$; in particular $y\neq 0$. Furthermore if $y^* = y$ then $y^2 = y$. Finally, if $\tau_t(Q) = Q$ then $\tau_t(y) = y$.
\end{lemma}
\begin{proof}
Applying $(\psi\otimes\id)$ to \eqref{gleq} and using   right invariance of $\psi$ we obtain \[\psi(Q)Q = \psi(y)Q\] thus $\psi(Q)= \psi(y)$. 
The identity  \begin{equation}\label{left2}(\I\otimes y) \Delta(y) = y\otimes y \end{equation}
is the consequence of the following computation
\begin{align*}
  ((\I\otimes y)\Delta(y))\otimes Q & =(\I\otimes y\otimes\I)(\Delta\otimes \id)(y\otimes Q)\\ & =(\I\otimes y\otimes Q)(\Delta\otimes\id)(\Delta(Q)) \\&= (\I\otimes\I\otimes Q)(\I\otimes\Delta(Q))(\id\otimes\Delta)(\Delta(Q))\\&=  (\I\otimes\I\otimes Q)(\id\otimes\Delta)( (\I\otimes Q)\Delta(Q))\\&=  (\I\otimes\I\otimes Q)(\id\otimes\Delta)( y\otimes Q)\\&=(y\otimes y)\otimes Q.
\end{align*}

If $y^* = y$ then 
\begin{align*}(\I\otimes Q)\Delta(Q) &= y\otimes Q\\&=(y\otimes Q)^* \\&= \Delta(Q)(\I\otimes Q).\end{align*} In particular $(\I\otimes Q)\Delta(Q)$ is a projection and thus  $(Q\otimes y)$ and $y$ are projections.

Applying $\tau_t\otimes\tau_t$ to $(\I\otimes Q)\Delta(Q) = y\otimes Q$ we see that $\tau_t(y) \otimes Q = y \otimes Q$, i.e. $y = \tau_t(y)$. 
\end{proof}

\begin{definition}\label{shiftsdef}
Let $Q\in\Linf(\GG)$ be  self-adjoint  projections and $P$   a group-like projection.  We say that $Q$ is a right-shift of $P$ if $\Delta(Q)(\I\otimes Q) = P\otimes Q$.  We say that $Q$  is a left shift of $P$ if $\Delta(Q)(Q\otimes\I) = Q\otimes P$.
\end{definition}
\begin{remark}\label{remarkleft}
Suppose that $Q,P$ are self-adjoint projection such that  $\Delta(Q)(Q\otimes\I)=Q\otimes P$. Applying  Lemma \ref{basic_lem}  to $R(Q)$ we see that $R(P)$ is a group-like projection. If   $Q$ is preserved by $\tau_t$ then  $R(P)$ is preserved by $\tau$  as well. In particular  $R(P) = P$  (see Lemma \ref{remR}) and we get $\Delta(P)(\I\otimes P) = P\otimes P$. 
\end{remark}
\section{From contractive idempotents to shifts of group-like projections}\label{Sec3}
The theory of contractive idempotent functionals on coamenable locally compact quantum groups were developed in \cite{NSSS}. Actually most of the results  proved there (in particular those contained in Sections 1-3)  hold without coamenability; the universal way to drop the coamenability is to repeat  essentially the same proof  as in \cite{NSSS}, with the difference that whenever the multiplicative unitary $\ww$ is used, one  must replace it with its half-lifted version $\wW$. In Section \ref{sec6}  we shall show how to drop the amenability which was used in the proof of \cite[Theorem 4.1]{NSSS}.   

Let $\omega\in\C_0^u(\GG)^*$ be a contractive idempotent functional and $|\omega|_r, |\omega|_l\in\C_0^u(\GG)^*$ the idempotent states assigned to $\omega$ as described in \cite[Theorem 2.4]{NSSS}, where $|\omega|_r$ ($|\omega|_l$) is the right (respectively the left) absolute value of $\omega$. For $\mu\in\{\omega,\omega^*,|\omega|_l,| \omega|_r\}$ we define $P_\mu\in\M(\C_0(\hh\GG)) $ and $E_\mu:\Linf(\GG)\to \Linf(\GG)$ by
\begin{align*}
P_\mu &= (\id\otimes \mu)(\wW),\\
E_\mu(x)&=\mu\staru x\coloneqq (\id\otimes\mu)(\wW(x\otimes\I)\wW^*).
\end{align*}
Note  that $P_\mu$ being a  projection of norm  not grater than $1$, it  must be of norm one and thus self-adjoint. 
We shall use the notation $\Delta^{r,u}(x) = \wW(x\otimes\I)\wW^*$. It can be checked that $\Delta^{r,u}$ defines a morphism $\Delta^{r,u}\in\Mor(\C_0(\GG),\C_0(\GG)\otimes\C^u_0(\GG))$ satisfying \begin{equation}\label{poddel}\C_0(\GG)\otimes\C^u_0(\GG) = \{\Delta^{r,u}(a)(b\otimes \I):a,b\in\C_0(\GG)\}^{\textrm{cls}}.\end{equation}
 
 \begin{theorem}\label{thmbis}
 Let $\omega\in\C_0^u(\GG)^*$ be a contractive idempotent state and $P_\omega,P_{|\omega|_r},P_{|\omega|_l}\in\Linf(\hh\GG)$ the self-adjoint projections introduced above. Then, $\omega$ is preserved by $\tau^u$ and $P_\omega$ is $\hh\tau$-invariant left (right) shift of $P_{|\omega|_r}$ ($P_{|\omega|_l}$). 
 \end{theorem}
 \begin{proof}
We have, (see \cite[Lemma 3.1]{NSSS})
\begin{align}
  \label{eqser1}  E_\omega(E_\omega(a)b) &= E_\omega(a)E_{|\omega|_l}(b),\\
    \label{eqser2} E_{|\omega|_l}(E_{\omega^*}(a)b) &= E_{\omega^*}(a)E_{\omega}(b),\\
     E_\omega(aE_\omega(b)) &= E_{|\omega|_r}(a)E_{\omega}(b),\\E_{|\omega|_r}(aE_{\omega^*}(b)) &= E_{\omega}(a)E_{\omega^*}(b)
\end{align} for all  $a,b\in\Linf(\hh\GG)$. 
In what follows we shall write  $x_\nu = (\nu\otimes\id)(\ww)\in\Linf(\GG)$ for $\nu\in\Linf(\hh\GG)_*$. Note that \begin{align}\label{eq5}x_{\nu_1}x_{\nu_2} &= x_{\nu_2*\nu_1},\\\label{eq6} E_\mu(x_\nu) &=x_{P_\mu\nu}. \end{align}
Using \eqref{eqser1}, \eqref{eq5} and \eqref{eq6} we get
\begin{equation}\label{pomoc}P_\omega(\nu_2*(P_\omega\nu_1)) =(P_{|\omega|_l}\nu_2)*(P_{\omega}\nu_1) \end{equation} for all $\nu_1,\nu_2\in\Linf(\hh\GG)_*$. 
 Applying \eqref{pomoc} to the unit $\I\in\Linf(\hh\GG)$ we get \[(\nu_2\otimes\nu_1)(\hh\Delta(P_\omega)(\I\otimes P_\omega)) =(\nu_2\otimes\nu_1)( P_{|\omega|_l}\otimes P_{\omega}).
 \] Since the latter holds for all $\nu_1,\nu_2\in\Linf(\GG)_*$ we get 
 \[\hh\Delta(P_\omega)(\I\otimes P_\omega) = P_{|\omega|_l}\otimes P_{\omega}\] and we see that $P_\omega$ is a right shift of $P_{|\omega|_l}$.
 Similarly we check that  
 \begin{align}
 \label{eq2.8} \hh\Delta(P_{|\omega|_l})(\I\otimes P_{\omega^*}) &= P_\omega\otimes P_{\omega^*},\\
 \hh\Delta(P_\omega)(P_\omega\otimes \I) &=  P_{\omega}\otimes P_{|\omega|_r},\\\label{eq3}\hh\Delta(P_{|\omega|_r})(P_{\omega^*}\otimes \I) &= P_{\omega^*}\otimes P_{\omega}.\end{align} 
 In particular $P_\omega$ is a left shift of $P_{|\omega|_r}$. 
Since $\omega^*$ (recall that $\omega^*(a) = \overline{\omega(a^*)}$) is a contractive idempotent state and $|\omega^*|_l = |\omega|_r$ and $|\omega^*|_r = |\omega|_l$ we also have (c.f. Equation \eqref{eq2.8})
 \begin{equation}\label{eq4}\hh\Delta(P_{|\omega|_r})(\I\otimes P_\omega) = P_{\omega^*}\otimes P_\omega\end{equation}
In particular using  Equation \eqref{relPN3} together with \eqref{eq3} and \eqref{eq4}  we see that 
 $P_\omega$ is a minimal central projection in $\widetilde{N_{|\omega|_r}}$. Using  Lemma \ref{invar} and Lemma \ref{leminvt} we conclude that $\hh\tau_t(P_\omega) = P_\omega$.
Using $(\hh\tau_t\otimes\tau_t^u)(\wW) = \wW$ we get $\omega = \omega\circ\tau^u_t$ for all $t\in\RR$. 
\end{proof}
\begin{remark}
Let $\omega\in\C_0^u(\GG)^*$   be a contractive idempotent functional and $v\in\C^u_0(\GG)$ the element  satisfying $\omega  =  |\omega|_lv   =  v|\omega|_r$ (c.f. \cite[Theorem 2.3]{NSSS}).  Since the functionals $\omega,|\omega|_l,  |\omega|_r$ are preserved by $\tau^u$ we conclude that  $\tau_t^u(v) = v$  for all $t\in\RR$. 
\end{remark}

 \section{From shift of group-like projections to contractive idempotent functionals}\label{Sec4}
The next theorem in particular provides the converse of Theorem \ref{thmbis}.
 \begin{theorem}\label{thmcis}
Let $Q\in\Linf(\hh\GG)$ be a left shift of a group-like projection $P$ \[(Q\otimes \I) \hh\Delta(Q)=Q\otimes P\] satisfying $\hh\tau_t(Q) = Q$. Then 
\begin{itemize}
    \item $\sigma^{\hh\varphi}_t(Q) = Q$ and $\sigma^{\hh\psi}_t(Q) = Q$;
    \item $\hh\Delta(P)(\hh R(Q)\otimes\I) = \hh R(Q)\otimes Q$;
    \item $\hh \Delta(Q)(\I\otimes P) = Q\otimes P$.
\end{itemize}
Moreover  there exists a contractive idempotent functional $\omega\in\C^u_0(\GG)^*$ such that $Q = (\id\otimes\omega)(\wW)$. 
\end{theorem}
\begin{proof}
Using Remark \ref{remarkleft} we see that $P\neq 0$ is a group-like projection preserved by $\hh\tau$. 
The strong left invariance of $\hh\varphi$ has the form 
\[\hh S\left((\id\otimes\hh\varphi)(\hh\Delta(a^*)(\I\otimes b)) \right)= (\id\otimes\hh\varphi)((\I\otimes a^*)\hh\Delta(b))\] where $a,b\in\Linf(\hh\GG)$ satisfy $\hh\varphi(a^*a),\hh\varphi(b^*b)<\infty$. 
Thus we have 
\begin{align*}
    \hh S\left((\id\otimes\hh\varphi)(\hh\Delta(a^*)\hh\Delta(P)(\hh R(Q)\otimes b)) \right)&=Q\hh S\left((\id\otimes\hh\varphi)(\hh\Delta(a^*P)(\I\otimes b)) \right)\\&=
    (\id\otimes\hh\varphi)((\I\otimes a^*)(Q\otimes P)\hh\Delta(b)) \\&=
    (\id\otimes\hh\varphi)((Q\otimes a^*)\hh\Delta(Qb)) \\&= \hh S\left((\id\otimes\hh\varphi)(\hh\Delta(a^*)(\hh R(Q)\otimes Q  b)) \right)
\end{align*} and we get 
\begin{equation}\label{identyx}\hh\Delta(P)(\hh R(Q)\otimes \I) =\hh R(Q)\otimes Q.\end{equation}
Let $\sN$ be the coideal $\hh\GG$-generated by $P$ and let $\nu\in\Linf(\hh\GG)_*$ be such that, $\|\nu\|=1$ and $\nu(\hh R(Q)) = 1$. Then, using Equation \eqref{identyx}  we get  \begin{equation}\label{xslice}Q=(\nu\otimes\id)(\hh\Delta(P)(\hh R(Q)\otimes \I))\in\sN.\end{equation} Since   $\sN = \{a\in\Linf(\hh\GG):\hh\Delta(a)(\I\otimes P) = a\otimes P\}$ (see \ref{relPN4}) we have $\hh\Delta(Q)(\I\otimes P) = Q\otimes P$. 

Applying $\hh R\otimes \hh R$ to Equation \eqref{identyx} we get $\hh\Delta(P)(\I\otimes Q) = \hh R(Q)\otimes Q$. In particular $Q\in Z(\sN)$ and it is a minimal projection in $\sN$ (c.f. Equation \eqref{relPN3}).  Using  Lemma \ref{invar} and Corollary \ref{corinvsig}   we conclude that $\sigma^{\hh\psi}_t(Q) = Q=\sigma^{\hh\varphi}_t(Q)$ for all $t\in\RR$. 

Using \eqref{xslice} we see that there exists a functional $\mu\in\Linf(\hh\GG)_*$, $\|\mu\| \leq 1$ such that $(\mu\otimes\id)(\hh\Delta(P))=Q$ (put $\mu = \hh R(Q)\nu$). Let $\tilde\omega\in\C_0^u(\GG)^*$ be the idempotent state such that $(\id\otimes \tilde\omega)(\wW) = P$. Then 
\[Q = (\mu\otimes\id)((\id\otimes\id\otimes\tilde\omega)(\wW_{23}\wW_{13})) = (\id\otimes(a\cdot\tilde\omega))(\wW)\] where $a = (\mu\otimes\id)(\wW)\in\C_0^u(\GG)$,  $\|a\|\leq 1$. Putting $\omega = \tilde\omega a$ we see that $\|\omega\|\leq 1$ and $(\id\otimes\omega)(\wW)= Q$. Since $Q$ is a projection $\omega$ must be an idempotent functional. 
\end{proof}

\begin{remark}\label{compar}
Left shifts of group-like projections considered in \cite{UnPr1} were assumed to satisfy $\hh\varphi(Q)=\hh\varphi(P)<\infty$ and they were defined by two conditions:\begin{equation}\label{unprdef}\hh\Delta(Q)(\I\otimes P) = Q\otimes P,\quad \hh\Delta(P)(\I\otimes Q) = \hh R(Q)\otimes Q.\end{equation} 
The $\hh\tau$ - invariance of $Q$  then follows. 

Note that assuming in Theorem \ref{thmcis} that $Q$ satisfies $\hh\Delta(Q)(Q\otimes\I) = Q\otimes P$ and $\hh\tau_t(Q) =Q$  for all $t\in\mathbb{R}$,  we get  \eqref{unprdef} and $\hh\varphi(P) = \hh\varphi(Q)$  as a consequence. Conversely, assuming $\hh\Delta(P)(\I\otimes Q) = \hh R(Q)\otimes Q$ and $\hh\tau_t(Q) = Q$ and using strong left invariance of $\hh\varphi$ as in the computation preceding \eqref{identyx} we can prove that $\hh\Delta(Q)(Q\otimes\I) = Q\otimes P$. 
\end{remark}
\begin{corollary}\label{corimport}
If $Q$ satisfies the assumptions of Theorem \ref{thmcis} then 
\[\overline{\{(\mu\otimes\id)(\hh\Delta(P)):\mu\in\Linf(\hh\GG)_*\}}^{\sigma-\textrm{weak}} = \overline{\{(\mu\otimes\id)(\hh\Delta(Q)):\mu\in\Linf(\hh\GG)_*\}}^{\sigma-\textrm{weak}}.\]
 Moreover there exists a group-like projection $\tilde{P}\in\Linf(\hh\GG)$ such that $\hh\Delta(Q) (\I\otimes Q) = \tilde{P}\otimes Q$.  
\end{corollary}
\begin{proof}The existence of $\tilde{P}$ follows  from  Theorem  \ref{thmcis} and Theorem \ref{thmbis}.
Denote 
\[\sM =\overline{ \{(\mu\otimes\id)(\hh\Delta(P)):\mu\in\Linf(\hh\GG)_*\}}^{\sigma-\textrm{weak}}.\] 
Clearly, since $Q\in \sM$ we have  $\overline{\{(\mu\otimes\id)(\hh\Delta(Q)):\mu\in\Linf(\hh\GG)_*\}}^{\sigma-\textrm{weak}}\subset\sM$. Taking $\nu\in\Linf(\hh\GG)_*$   such that $\nu(Q) = 1$ we get  \[P = (Q\nu\otimes \id)(\hh\Delta(Q))\in\overline{\{(\mu\otimes\id)(\hh\Delta(Q)):\mu\in\Linf(\hh\GG)_*\}}^{\sigma-\textrm{weak}}\] and   using Equation \eqref{relPN3} conclude the converse containment  \[\sM \subset \overline{\{(\mu\otimes\id)(\hh\Delta(Q)):\mu\in\Linf(\hh\GG)_*\}}^{\sigma-\textrm{weak}}.\]
\end{proof}
\section{Coideals admitting an atom}\label{Sec5}
\begin{theorem} Let $\sN\subset \Linf(\GG)$ be a coideal   and assume that its codual coideal $\tilde{\sN}\subset \Linf(\hh\GG)$ admits a non-zero  minimal central projection $ Q\in Z(\widetilde\sN)$  such that $\hh\tau_t(Q) = Q$.  
Then there exists a unique idempotent state $\omega\in\C_0^u(\GG)^*$ such that $Q$ is a left shift of $P_\omega$. Moreover $\sN = \sN_\omega$ and we have $\widetilde\sN = \{(\mu\otimes\id)(\hh\Delta(Q)):\mu\in\Linf(\hh\GG)_*\}^{\sigma-\textrm{cls}}$. 
\end{theorem}
\begin{proof} 
 Let $a\in\widetilde\sN$. By minimality and centrality of $Q\in\widetilde\sN$ there exists $\breve{a}\in\Linf(\hh\GG)$ such that 
    \[\hh\Delta(a)(\I\otimes Q) = \breve{a}\otimes Q.\]   Note that  $\sM:=\{\breve{a}:a\in\widetilde\sN\} \subset \Linf(\hh\GG)$ forms a coideal and $\pi:\widetilde\sN\ni a\mapsto \breve a\in\sM$ is $\hh\GG$ - equivariant normal $*$-homomorphism which commutes with $\hh\tau$. E.g. the $\hh\GG$ - equivariance follows from the computation 
    \begin{align*}(\hh\Delta\otimes \id)(\pi(a)\otimes Q)  &= (\hh\Delta\otimes \id)(\hh\Delta(a)(\I\otimes Q)\\& =(\id\otimes \hh\Delta)(\hh\Delta(a))(\I\otimes \I\otimes Q) \\& = (\id\otimes\pi)(\hh\Delta(a))\otimes Q. \end{align*} Similarly we check that $\pi\circ\hh\tau_t = \hh\tau_t\circ\pi$. 
    By the ergodicity of the action of $\hh\GG$ on $\widetilde\sN$, $\pi$ is an isomorphism. Indeed otherwise there exists  $q\in Z(\widetilde\sN)$ such that $\ker\pi = \widetilde\sN q$, $q\neq 0$. Using the covariance of $\pi$ we get  $\hh\Delta(q) \leq  \I\otimes q$. Thus using \cite[Lemma 6.4]{KV} we get $q=\I$ which  contradicts  the unitality of $\pi$.  
    
 Let us denote $P_0 = \pi(Q)$.   Using Lemma \ref{basic_lem} we see that  $P_0$ is a group-like projection. Clearly,  $P_0$ is  minimal, central $\hh\tau_t$-invariant in $\sM$  (since the same holds for $Q\in\widetilde\sN$). 
 Let $\omega_0\in\C_0^u(\GG)^*$ be the idempotent state such that $P_0 = (\id\otimes\omega_0)(\wW)$ and $\sN_{\omega_0}\subset\Linf(\GG)$ the coideal assigned to $\omega_0$. 
 By the minimality and centrality of $P_0\in \sM$, for all $b\in\sM$ there exists $c\in\Linf(\hh\GG)$ such that  $\hh\Delta(b)(\I\otimes P_0) = c\otimes P_0$, i.e.
   \begin{equation}\label{auxthm}(\I\otimes b)\ww(P_0\otimes\I) = \ww(P_0\otimes c).\end{equation} Slicing Equation \eqref{auxthm} with $(\nu\otimes\id)$ where $\nu\in\Linf(\hh\GG)_*$ and using \eqref{relPN1} we see that  $bx = xc$ for all $x\in\sN_{\omega_0}$. Putting $x= \I$ we get $b = c$ and thus $\sM\subset \widetilde{\sN_{\omega_0}}$. Since $P_0\in\sM$ and $\widetilde{\sN_{\omega_0}}$ is $\hh\GG$-generated by $P_0$ (see Equation  \eqref{relPN3}) we also have the converse inclusion. In particular $\sM$ is $\hh\GG$-generated by $P_0$ and since $\pi$ identifies  $\widetilde\sN$ and $\sM$, $\widetilde\sN$\ is $\hh\GG$-generated by $Q$. 
 
The equality  $\pi(Q) = P_0$ reads $\hh\Delta(Q)(\I\otimes Q)= P_0 \otimes Q$  and together with  $\hh R(P_0) = P_0$ implies  $\hh\Delta(\hh R(Q))(\hh R(Q)\otimes\I) = \hh R(Q)\otimes P_0$.  Using Corollary \ref{corimport}   we see that there exists $\hh\tau$-invariant group-like projection  $P$ such that $\hh\Delta(\hh R(Q))(\I\otimes \hh R(Q)) =  P\otimes \hh R(Q)$. In particular we have $\Delta(Q)(Q\otimes \I) =Q\otimes  P$  and, again using Corollary \ref{corimport}, we see that $Q$ and $P$  $\hh\GG$-generate the same coideal, i.e. $\widetilde\sN = \{(\mu\otimes \id)(\hh\Delta(P):\mu\in\Linf(\hh\GG)_*\}^{\sigma-\textrm{weak}}$. Denoting by $\omega\in\C_0^u(\GG)^*$   the idempotent state corresponding to $P$ we get $\sN = \sN_\omega$ and we are done.  
\end{proof}
\section{Ternary rings of operators and contractive idempotent functionals}\label{sec6}
 
\begin{definition}
 Let $\sM$ be a von Neumann algebra, $\sX\subset\sM$   a linear subspace closed in the $\sigma-\textrm{weak}$ topology. We say that $\sX$ is a ternary ring of operators (TRO) if \[ab^*c\in\sX \textrm{ whenever } a,b,c\in\sX.\] Let $\psi$ be an n.s.f. weight on $\sM$ and $\eta$ the GNS map assigned to $\psi$. We say that $\sX$ is $\psi$-integrable (or simply integrable) if $\sX\cap D(\eta)$ is a $\sigma-\textrm{weakly}$ dense subspace of $\sX$. The Hilbert space completion of $\eta(\sX\cap D(\eta))$ will be denoted by $\Ltwo(\sX)$.  
\end{definition}

If  $\sX\subset \sM$ is a TRO then $\sX^*$ is a TRO. Note that $\{a^*b:a,b\in\sX\}^{\sigma-\textrm{cls}}\subset \sM$ is a subalgebra of $\sM$  which we shall denote  by $\langle \sX^*\sX\rangle$.  If   $\I_\sM\in \langle \sX^*\sX\rangle\cap\langle \sX\sX^*\rangle$ then we say that $\sX$ is a  non-degenerate TRO in $\sM$ (see \cite{SalSkTro} for the discussion of non-degeneracy condition in the contetxt of TRO's).  
\begin{example}\label{exx}
Let $\GG$ be a locally compact quantum group and $\omega\in\C_0^u(\GG)^*$ a contractive idempotent functional. Let $\sX_\omega = \{x\in\Linf(\GG):E_\omega( x) = x\}$.  Using the method of the proof of \cite[Theorem]{NSSS} we see that $\sX_\omega$ is TRO.   Since $\omega\staru x\in D(\eta)$ for all $x\in D(\eta)$ we conclude that $\sX_\omega$ is integrable.    Let us prove that $\sX_\omega$ is non-degenerate. For all $a,b\in\Linf(\GG)$ we have  (c.f. Equation \ref{eqser2})
\begin{equation}\label{eq51}E_\omega(a)^*E_\omega(b) =  E_{|\omega_l|}(E_\omega(a)^*b)).\end{equation} Using Equation \eqref{poddel} in the second equality below  we get  \[\{E_\omega(a)^*b  :a,b\in\C_0(\GG)\}^{\sigma-\textrm{cls}}   = \{(\id\otimes\omega^*)(\Delta^{r,u}(a^*)(b\otimes\I)):a,b\in\C_0(\GG)\}^{\sigma-\textrm{cls}} = \Linf(\GG)\]  and using Equation \eqref{eq51}  we conclude that $\{E_\omega(a)^*E_\omega(b):a,b\}^{\sigma-\textrm{cls}}$  is equal $\sN_{|\omega|_l}$. Similarly we check that $\langle\sX_\omega\sX_\omega^*\rangle = \sN_{|\omega|_r}$  and we see that $\sX_\omega$ is  nondegenerate. It is easy to check that $\sX_\omega$ is preserved  by $\GG$. Using $\tau^u$ - invariance of $\omega$ we see that $\tau_t\circ E_\omega= E_\omega\circ\tau_t$, and thus $\sX_\omega$ is preserved by $\tau$. 
\end{example} 
Let  $\sX\subset\sM$ be a TRO.  The linking von Neumann algebra $\sA_\sX\subset  \sM\vtens\M_2(\CC)$ is defined by \[\sA_\sX = \begin{bmatrix}\langle\sX\sX^*\rangle&\sX\\\sX^*&\langle\sX^*\sX\rangle\end{bmatrix}.\] Now suppose that $\sM = \Linf(\GG)$ and $\sX$ is  non-degenerate TRO which is preserved by $\GG$, i.e. $\Delta(\sX)\subset\Linf(\GG)\vtens\sX$.  Then $(\Delta\otimes\id):\Linf(\GG)\vtens\M_2(\CC)\to\Linf(\GG)\vtens\Linf(\GG)\vtens\M_2(\CC)$ restricts  to an action  of $\GG$ on $\sA_\sX\subset \Linf(\GG)\vtens\M_2(\CC)$. In particular  the   weak Podle\'s condition holds $\sA_\sX = \{(\omega\otimes\id)(\Delta\otimes\id)(a)):\omega\in\Linf(\GG)_*,a\in\sA_\sX\}^{\sigma-\textrm{cls}}$ (see \cite[Corollary 2.7]{gpext}). This implies the equality  \eqref{weakpodeq} in the next proposition.  
\begin{proposition}\label{prop_pod}
Let $\sX\subset \Linf(\GG)$ be a non-degenerate $\GG$-invariant TRO. Then we have \begin{equation}\label{weakpodeq}\sX = \{(\mu\otimes\id)(\Delta(a)):a\in\sX,\mu\in\Linf(\GG)_*\}^{\sigma-\textrm{cls}}.\end{equation}
If $\sX$ is preserved by $\tau$ then it is preserved by $\sigma^\psi$. If in addition $\sX$ is integrable then $\sX^*$ is integrable.  
\end{proposition}
\begin{proof}
 As explained in the paragraph preceding the formulation of Proposition \ref{prop_pod}, $\sX$ satisfies \eqref{weakpodeq}. Suppose that $\sX$ is preserved by $\tau$. Using the identity $(\sigma^\varphi_t\otimes\sigma^\psi_{-t})\circ\Delta = \Delta\circ\tau_t$ (see \cite[Theorem 6.8]{KV}) and Equation \eqref{weakpodeq} we see that $\sigma^\psi$ preserves $\sX$. Finally suppose that $\psi(x^*x)<\infty$ and $x\in D(\sigma^\psi_{\frac{i}{2}})$. Then $\sigma^\psi_{\frac{i}{2}}(x)\in\sX$ and $\psi(\sigma^\psi_{\frac{i}{2}}(x)\sigma^\psi_{\frac{i}{2}}(x)^*)=\psi(x^*x)<\infty$. Since the set $\left\{\sigma^\psi_{\frac{i}{2}}(x):x\in\sX\cap D(\sigma^\psi_{\frac{i}{2}})\right\}$ is dense in $ \sX$ we conclude that $\sX^*$ is integrable. 
\end{proof}
\begin{proposition}\label{propTro}
 Let $\GG$ be a locally compact quantum group and $\sX\subset \Linf(\GG)$ a non-degenerate, integrable TRO such that  $\Delta(\sX)\subset \Linf(\GG)\vtens\sX$. Let $Q\in\B(\Ltwo(\GG))$ be the orthogonal projection onto $\Ltwo(\sX)$ and $P\in\B(\Ltwo(\GG))$ the orthogonal projection onto $\{x^*\Ltwo(\sX):x\in\sX\}^{\textrm{cls}}$.  Then
 \begin{enumerate}
     \item $x^*Q = Px^*$ for all $x\in\sX$;
     \item $Q,P\in\Linf(\hh\GG)$  and $Q$ is a right shift of $P$.
 \end{enumerate}  
\end{proposition}
\begin{proof}
 Ad(1): Let $x_1\in\sX$ and $x_2\in\sX\cap D(\eta)$. Then $x_1^*\eta(x_2)\in P\Ltwo(\GG)$. Thus \begin{equation}\label{tro1}x^*_1Q\eta(x_2) =x^*_1\eta(x_2)= Px_1^*\eta(x_2).\end{equation} Suppose now that $\xi\in\Ltwo(\GG)$ satisfies $Q\xi = 0$. The latter holds iff 
 \[(\eta(x_1x_2^*x_3)|\xi) = 0\] whenever $x_3\in\sX\cap D(\eta)$ and $x_1,x_2\in\sX$, which in turn holds iff \[(x_2^*\eta(x_3)|x_1^*\xi) = 0.\] The latter holds if and only if $Px_1^*\xi = 0 $. Since   $0 = x_1^*Q\xi$ we get $x_1^*Q\xi = Px_1^*\xi$, which  together with   \eqref{tro1} yields $x^*_1Q = Px_1^*$ for all $x_1\in\sX$. 
 
 \noindent Ad(2): Let $y\in\Linf(\hh\GG)'$ be an element for which there exist $\mu\in\Linf(\GG)_*$, such that  $y\eta(z) = \eta((\mu\otimes\id)(\Delta(z))$ for all  $z\in D(\eta)$ (the set of such   $y$'s is dense in $\Linf(\hh\GG)'$, see e.g. \cite[Equation 4.4]{coid_sub_st}). For all $x\in\sX\cap D(\eta)$ we have $y\eta(x) = \eta((\mu\otimes\id)(\Delta(x))\in\Ltwo(\sX)$, i.e. $yQ = QyQ$. Since the latter holds for all elements of a dense set  of $\Linf(\hh\GG)'$ we conclude that  $Q\in\Linf(\hh\GG)$.  
 
Let us show that $Q$ is a left shift of $P$. 
For  $x\in D(\eta)\cap\sX$ and $\omega\in\Linf(
\GG)_*$ we get 
\begin{align*}(\id\otimes\omega)((\I\otimes Q)\ww)\eta(x)&=\eta((\id\otimes \omega)((\I\otimes Q)\Delta(x)))\\&=\eta((\id\otimes \omega)(\Delta(x)(\I\otimes P)))\\&=(\id\otimes\omega)(\ww(\I\otimes P))\eta(x),\end{align*} where in the second equality we used $(\I\otimes Q)\Delta(x) =\Delta(x)(\I\otimes P)$ (see Ad(1)).  This computation shows that $(\I\otimes Q)\ww(Q\otimes\I) = \ww (Q\otimes P)$ which is equivalent with   $\hh\Delta(Q)(\I\otimes Q) = P\otimes Q $. In particular $P\in\Linf(\hh\GG)$ and $Q$ is a right shift of $P$. 
\end{proof} 
\begin{corollary}\label{corom}
 Suppose that $\sX$ satisfies  the conditions of  Proposition \ref{propTro} and $\sX$ is preserved by $\tau$ and let $Q$ be the orthogonal projection onto $\Ltwo(\sX)$. Then there exists an idempotent contractive functional $\omega\in\C_0^u(\GG)^*$ such that $Q = (\id\otimes\omega)(\wW)$. 
\end{corollary}
 \begin{proof}
It is enough to check that $Q$  is preserved by the scaling group $\hh\tau$ (c.f. Theorem \ref{thmcis}). Using Proposition \ref{prop_pod}  we see that $\sX$ is preserved by $\sigma^\psi$. In particular $\nabla^{it}Q \nabla^{-it} = Q$  where  $T  =  J\nabla$ is  the Tomita-Takesaki operator assigned with $\psi$. Then  using  \cite[Proposition 2.1]{KVvN} we see that $\hh\tau_t(Q)  =  Q$ for all $t\in\RR$. 
 \end{proof}
\begin{theorem}\label{mainsec5}
 There is a one to one correspondence between contractive idempotent functionals and non-degenerate, integrable, $\tau$-invariant  TRO's   such that
  $\Delta(\sX)\subset\Linf(\GG)\vtens\sX$, where for a given contractive idempotent functional $\omega\in\C_0^u(\GG)^*$ the corresponding TRO is $\sX_\omega$. 
\end{theorem}
\begin{proof}
Let $\omega\in\C_0^u(\GG)$ be a contractive idempotent functional. Then $\sX_\omega$ described in Example \ref{exx} is a TRO satisfying the conditions of   Theorem \ref{mainsec5}. Conversely, let $\sX$ be a TRO satisfying the conditions of Theorem \ref{mainsec5}. Consider the linking von Neumann algebra $\sA_{\sX}\subset\Linf(\GG)\vtens\M_2(\CC)$. The weight $\psi\otimes\textrm{Tr}$ restricts to an n.s.f. weight on  $\sA_{\sX}$. Using Proposition \ref{prop_pod} we see that $\sA_\sX$  is preserved by $\sigma^\psi\otimes\id$.    Using \cite[Theorem IX.4.2]{Tak} we get a conditional expectation $E:\Linf(\GG)\vtens\M_2(\CC)\to \Linf(\GG)\vtens\M_2(\CC)$ onto $\sA_{\sX}$. 

Let us show that $E$ is  a  Shur map (we use here the terminology from the proof of \cite[Theorem 4.1]{NSSS} ), i.e. there exist  maps $E_{ij}:\Linf(\GG)\to\Linf(\GG)$ $i,j=1,2$, such that 
\begin{equation}\label{shur}E \begin{bmatrix} a&b\\
c&d\end{bmatrix}= \begin{bmatrix}E_{11}(a)&E_{12}(b)\\E_{21}(c)&E_{22}(d)\end{bmatrix}\end{equation}  for all $a,b,c,d\in\Linf(\GG)$.
 In order to do this let us note that   $\sA_{\sX}$ is preserved by   the      Shur  maps  of the form  \[\begin{bmatrix} L_{\nu_{11}}&L_{\nu_{12}}\\
L_{\nu_{21}}&L_{\nu_{22}}\end{bmatrix}:\Linf(\GG)\vtens\M_2(\CC)\to\Linf(\GG)\vtens\M_2(\CC)\] where $\nu_{ij}\in\Linf(\GG)_*$ for $i,j=1,2$ and 
\[\begin{bmatrix} L_{\nu_{11}}&L_{\nu_{12}}\\
L_{\nu_{21}}&L_{\nu_{22}}\end{bmatrix}\begin{bmatrix} a&b\\
c&d\end{bmatrix} =  \begin{bmatrix} a*\nu_{11}&b*\nu_{12}\\
c*\nu_{21}&d*\nu_{22}\end{bmatrix}\] for all $a,b,c,d\in\Linf(\GG)$. Using the identification $\Ltwo(\Linf(\GG)\vtens\M_2(\CC))  = \bigoplus_{i=1}^4\Ltwo(\GG)$ we see that the orthogonal projection   $P$  onto $\Ltwo(\sA_\sX)\subset\bigoplus_{i=1}^4\Ltwo(\GG)$  satisfies $yP   = PyP$  for all $y\in\bigoplus_{i=1}^4\Linf(\hh\GG)'$ acting diagonally on  $\bigoplus_{i=1}^4\Ltwo(\GG)$.    Since $\bigoplus_{i=1}^4\Linf(\hh\GG)'$ is a von Neumann  algebra we get  $yP = Py$ for  all $y\in\bigoplus_{i=1}^4\Linf(\hh\GG)'$ and thus  $P$  acts diagonally on 
 $\bigoplus_{i=1}^4\Ltwo(\GG)$. Since $P$ is a Hilbert space version of $E$ we  conclude that $E$ is a Shur map of the form \eqref{shur}. The Hilbert space version of $E_{12}$  is the orthogonal projection  $Q:\Ltwo(\GG)\to\Ltwo(\GG)$ onto $\Ltwo(\sX)$, as described in Proposition \ref{propTro}, and we have \[\eta(E_{12}x) = Q\eta(x)\] for all $x\in D(\eta)$.  Since $\sX$ is preserved by $\tau$, we can  use  Corollary \ref{corom}   to  conclude that there exists a  contractive idempotent functional $\omega\in\C_0^u(\GG)^*$ such that $Q = (\id\otimes\omega)(\wW)$. Using Equation \eqref{multunit2} we see that   $ Q\eta(x) = \eta(\omega\staru x)$ and we get $E_{12}(x) = \omega\staru x$ for all $x\in\Linf(\GG)$. This shows that  $\sX = \sX_\omega$  and we are  done. 
\end{proof}
\begin{remark}
A similar result linking TRO's with contractive idempotent functionals (in the $\C^*$-algebraic framework) was also obtained in \cite[Theorem 4.1]{NSSS}. In the corresponding reasoning  amenability was assumed and used   to show that $E_{12}$ has the form $E_{12}(x) = \omega\staru x$. We were able to avoid  amenability in our reasoning by proving   Proposition \ref{propTro} first and then using it to get the same formula for $E_{12}$.
\end{remark}
 \subsection*{Acknowledgements} The author wishes to thank Adam Skalski for helpful comments.
The author was partially supported by the NCN (National Center of Science) grant
 2015/17/B/ST1/00085. 


\begin{thebibliography}{66}


\bibitem{Cohen} P.J.~Cohen. On a conjecture of Littlewood and idempotent measures. \emph{Amer.~J.~Math.} \textbf{82} (1960), 191--212.

 

\bibitem{FallKasp} R.~Fall and P.~ Kasprzak. Group-like projections for locally compact quantum groups. \emph{Journal of Operator Theory, to appear.}

\bibitem{FranzLeeSkalski}U.~Franz, H.H.~Lee and A.~Skalski. Integration over the quantum diagonal subgroup and associated Fourier-like algebras.
 \emph{Int.~J.~Math.} \textbf{27}, 1650073 (2016).
 

\bibitem{FSO} U.~Franz and A.~Skalski. On idempotent states on quantum groups. \emph{J.~Algebra}, \textbf{322} (2009), 1774--1802.

\bibitem{FST} U.~Franz, A.~Skalski and R.~Tomatsu. Idempotent states on compact quantum groups and their classification on $U_q(2)$, $SU_q(2)$ and $SO_q(3)$. \emph{J.~Noncommut.~Geom.} \textbf{7} (2013), 221--254.

\bibitem{Host} B.~Host. La th\'eor\`eme des idempotents dans $\B(G)$. \emph{Bull.~Soc.~Math.~France} \textbf{114} (1986), 215--223.

\bibitem{IS} M.~Ilie and N.~Spronk. Completely bounded homomorphisms of the Fourier algebras. \emph{J.~Funct.~Anal.} \textbf{225} (2005), 480--499.

\bibitem{UnPr} C.~Jiang, Z.~ Liu, and J.~Wu. Noncommutative uncertainty principles. \emph{J.~Funct.~Anal.}
\textbf{270}(1) (2016), 264--311.

\bibitem{UnPr1} C.~Jiang, Z.~ Liu, and J.~Wu. Uncertainty principles for locally compact quantum
groups. \emph{J.~Funct.~Anal.}
\textbf{274}(1) (2018), 2399--2445.



\bibitem{KKS} M.~Kalantar, P.~Kasprzak, and A.~Skalski. Open quantum subgroups of locally compact quantum groups.  \emph{Adv.~Math.} \textbf{303} (2016),   322--359.

\bibitem{coid_sub_st} P.~Kasprzak and F.~Khosravi. Coideals, quantum subgroups and idempotent states.  \emph{Q.~J.~Math},
\textbf{68}(2) (2017), 583--615.

\bibitem{int} P.~Kasprzak, F.~Khosravi, and P.M.~So{\l}tan.  Integrable actions and quantum subgroups. \emph{International Mathematics Research Notices, to appear.}

\bibitem{PSL} P.~Kasprzak, and P.M.~So{\l}tan. Lattice of idempotent states on locally compact quantum groups. Preprint \texttt{arXiv:1802.03953 [math.OA]}.

\bibitem{embed} P.~Kasprzak, and P.M.~So{\l}tan. Embeddable quantum homogeneous spaces. \emph{J.~Math.~Anal.~Appl.} \textbf{411} (2014), 574--591.

\bibitem{gpext} P.~Kasprzak, and P.M.~So{\l}tan. Quantum groups with projection on von Neumann algebra level. \emph{J.~Math.~Anal.~Appl.}  \textbf{427}   (2015),  289--306.
 

\bibitem{KV} J.~Kustermans, and S.~Vaes: Locally compact quantum groups. \emph{Ann.~Scient.~\'{E}c.~Norm.~Sup.} $4^{\text{\tiny e}}$ s\'{e}rie, t.~\textbf{33} (2000), 837--934.

\bibitem{KVvN} J.~Kustermans, and S.~Vaes. Locally compact quantum groups in the von Neumann algebraic setting. \emph{Math.~Scand.} \textbf{92} (2003), 68--92.

 
\bibitem{NSSS} M.~Neufang, P.~Salmi, A.~Skalski, and  N.~Spronk. Contractive idempotent states on locally compact quantum groups. \emph{Indiana Univ.~Math.~J.}
\textbf{62}(6) (2013),  1983--2002.
   
\bibitem{PodlesSphere} P.~Podle\'s. Quantum spheres. \emph{Lett.~Math.~Phys.} \textbf{14}(3) (1987), 193--202.

\bibitem{Salmi_Survey} P.~Salmi. Idempotent states on locally compact groups and quantum groups \emph{Algebraic Methods in
Functional Analysis, The Victor Shulman Anniversary Volume} Operator Theory: Advances and
Applications, Vol. 233, I.G. Todorov and L. Turowska (Eds.), pp. 155--170, Birkh{\"a}user/Springer,
Basel, 2014.

\bibitem{SaS} P.~Salmi, and A.~Skalski. Idempotent states on locally compact quantum groups II. \emph{Q.~J.~Math.},
\textbf{68}(2) (2017), 421--431.

\bibitem{SalSkTro} P.~Salmi, and A.~Skalski. Actions of locally compact  (quantum) groups on
ternary rings of operators, their crossed products
and generalized Poisson boundaries. \emph{Kyoto J.~Math.} \textbf{57}(3) (2017), 667-691.

\bibitem{Tak} M.~Takesaki : \emph{Theory of operator algebras II}, Encyclopaedia Math. Sci. 125, SpringerVerlag,  Berlin 2003.

\end{thebibliography}
\end{document}